\documentclass[12pt,a4paper]{amsart}

\usepackage{amsmath,amsfonts,amssymb}

\usepackage[hmargin=2cm,vmargin=2cm]{geometry}

\usepackage[hyperfootnotes=false,colorlinks=true,linkcolor=blue,%
citecolor=purple,filecolor=magenta,urlcolor=cyan]{hyperref}
\usepackage{nameref,zref-xr}

\usepackage{amsthm,amscd}
\usepackage{epsfig}
\usepackage{color}
\usepackage[all]{xy}
\usepackage{tikz}
\usepackage{graphics, setspace}
\usepackage{breqn}
\usepackage{amstext}
\usepackage{array}

\newcommand{\CC}{{\mathbb{C}}}

\newcommand{\PP}{{\mathbb{P}}}

\newcommand{\F}[2]{{F_{#1}^{#2}}}

\def\A{{\mathcal A}}

\def\F{{\mathcal F}}

\def\H{{\mathcal H}}
\def\I{{\mathcal I}}

\def\O{{\mathcal O}}

\def\T{{\mathcal T}}

\def\res{{\mathrm{res}}}

\newcommand{\p}{{\partial}}

\newcommand{\mbC}{\mathbb C}

\def\d{{\partial}}

\newcommand{\mcT}{\mathcal{T}}

\newcommand{\bt}{{\bf t}}

\newtheorem{theorem}{Theorem}[section]
\newtheorem*{theorem*}{Theorem}
\newtheorem{proposition}[theorem]{Proposition}
\newtheorem{lemma}[theorem]{Lemma}
\newtheorem{corollary}[theorem]{Corollary}
\newtheorem*{corollary*}{Corollary}

\newtheorem{example}[theorem]{Example}
\newtheorem{remark}[theorem]{Remark}

\usepackage{color}

\def\&{\vspace{-5pt}&}

\makeatletter
\newsavebox{\@brx}
\newcommand{\llangle}[1][]{\savebox{\@brx}{\(\m@th{#1\langle}\)}%
  \mathopen{\copy\@brx\kern-0.5\wd\@brx\usebox{\@brx}}}
\newcommand{\rrangle}[1][]{\savebox{\@brx}{\(\m@th{#1\rangle}\)}%
  \mathclose{\copy\@brx\kern-0.5\wd\@brx\usebox{\@brx}}}
\makeatother

\allowdisplaybreaks

\numberwithin{equation}{section}

\begin{document}
\title{Open extension of a genus zero Hurwitz-Frobenius manifold}
\date{\today}
\author{Alexey Basalaev}
\address{A. Basalaev:\newline Faculty of Mathematics, HSE University, Usacheva str., 6, 119048 Moscow, Russian Federation}
\email{abasalaev@hse.ru}
\author{Anton Rarovskii}
\address{A.Rarovskii:\newline Faculty of Mathematics, HSE University, Usacheva str., 6, 119048 Moscow, Russian Federation and
Scientific School named after I.M. Krichever, Skolkovo Institute of Science and Technology, Bolshoy Boulevard 30 bld. 1, 121205 Moscow, Russian Federation
}
\email{ararovskiy@hse.ru}

\date{\today}

 \begin{abstract}
 Every Dubrovin--Frobenius manifold provides a solution to the WDVV equation, that was initially formulated in the study of the moduli space of curves. During the last two decades special attention was given to the "open" version of the moduli space of curves. Its genus zero intersection theory is governed by the system of equations called \textit{open WDVV equation}, that extends "classical" WDVV equation.

 In this note we study Dubrovin--Frobenius manifold structures on the Hurwitz spaces of genus zero. We construct explicitly the solutions to open WDVV equation for these Dubrovin--Frobenius manifolds and show that these solutions appear as the restriction of the Dubrovin--Frobenius structure to its discriminant.
 \end{abstract}
 \maketitle

\section{Introduction}
Dubrovin--Frobenius manifolds, moduli spaces and integrable systems are known to be closely related to each other since the early 90s (cf. \cite{D2, DVV, FSZ, Z} etc.). WDVV equation was a cornerstone playing an important role in all three topics. The function $F^c = F^c(t_1,\dots,t_n)$ is a solution to WDVV equation with the metric $\eta$ if for any fixed $1 \le \alpha,\beta,\gamma,\delta \le n$ holds
\begin{align*}
\sum_{\mu,\nu=1}^n \frac{\d^3 F^c}{\d t_\alpha\d t_\beta\d t_\mu}\eta^{\mu\nu}\frac{\d^3 F^c}{\d t_\nu\d t_\gamma \d t_\delta}
=
\sum_{\mu,\nu=1}^n \frac{\d^3F^c}{\d t_\gamma\d t_\beta\d t_\mu}\eta^{\mu\nu}\frac{\d^3 F^c}{\d t_\nu\d t_\alpha \d t_\delta}.
\end{align*}
This equation holds true for the generating functions of the intersection number of the Deligne-Mumford compactification of the moduli space of curves. Geometric framework for the study of WDVV equation was given by B. Dubrovin.
\\
\\
\noindent
\textbf{Dubrovin--Frobenius manifolds} were introduced by B.~Dubrovin in the early 90s (cf. \cite{D2}). This is a complex manifold $M$ equipped with an associative and commutative product $\circ: \T_M \otimes_M \T_M \to \T_M$ on the holomorphic tangent sheaf, a flat pairing $\eta: \T_M \otimes_M \T_M \to \O_M$ and a flat unit vector field. These data all together should satisfy the certain integrability condition. In particular, there should exist the special coordinates $t_1,\dots,t_n$, called \textit{flat} and a special function $\F \in \O_M$, called \textit{potential}, such that the unit vector field coincides with $\dfrac{\p}{\p t_1}$ and
\[
    \eta( \frac{\p}{\p t_\alpha},\frac{\p}{\p t_\beta}) = \frac{\p^3 \F}{\p t_1 \p t_\alpha \p t_\beta }  = \eta_{\alpha\beta} \in \CC,
    \quad
    \frac{\p}{\p t_\alpha} \circ \frac{\p}{\p t_\beta} = \sum_{\gamma,\delta=1}^n \frac{\p^3 \F}{\p t_\alpha \p t_\beta \p t_\gamma} \eta^{\gamma,\delta} \frac{\p}{\p t_\delta}.
\]

Associativity of the product $\circ$ is equivalent to WDVV equation.
Sometimes the Dubrovin--Frobenius manifold itself is built starting from the potential given --- like it is in Gromov--Witten theory. In this case the formulae above should be read from the right to the left.
In the ``geometric'' cases the pairing $\eta$ and the product $\circ$ are defined, but the potential and the flat coordinates have to be found.
Examples of such Dubrovin--Frobenius manifolds are Saito--Frobenius (cf. \cite{ST}) and Hurwitz--Frobenius manifolds.

\subsection*{Hurwitz--Frobenius manifolds} It was observed by B.~Dubrovin \cite[Lecture 5]{D2} that the space $\H_{g,K}$ of ramified coverings of $\PP^1$ by a genus $g$ Riemann surface with the prescribed ramification profile $K$ can be endowed with a structure of Dubrovin--Frobenius manifold. Such structures were investigated in \cite{S,M17,B14,PS} etc..
B. Dubrovin constructs $\eta$ and $\circ$ via the certain residue calculus that we introduce later.

In this text we consider the space $\H_{0;K}$ with $g=0$ and $K = \{k_0,\dots, k_N\}$. Its points are the functions $\lambda: \PP^1 \to \PP^1$
\[
 \lambda(z) := \frac{z^{k_0}}{k_0} + \sum _{a=1}^{k_0-1} v_{0,a} z^{a-1} + \sum_{i=1}^N\sum _{a=1}^{k_i} \frac{v_{i,a}}{ \left(z-v_{i,0}\right)^{a} }
\]
with arbitrary complex coefficients $v_{i,a}$ such that $v_{i,0} \neq v_{j,0}$ for $i \neq j$ and $v_{i,k_i} \neq 0$. The complex parameters $v_{i,a}$ are the coordinates of $\H_{0;K}$. The coordinates $v_\bullet$ look essential, however these are not flat. The flat coordinates $t_{i,\alpha} = t_{i,\alpha}(v_{i,\bullet})$ are obtained via the certain residue calculus involving the function $\lambda$. In particular, we have $v_{i,0} = t_{i,0}$ and $v_{i,k_i} = (t_{i,k_i}/k_i)^{k_i}$, while the formulae for the other flat coordinates are much more complicated.

This is well-known that the Dubrovin--Frobenius manifold structure of $\H_{0, K}$ has poles alongside the hypersurfaces $t_{i,0} = t_{j,0}$ and $t_{i,k_i} = 0$. 
Denote by $\F^K = \F^K(t_{i,\alpha})$ the potential of $\H_{0,K}$. Then this function can not be generally extended to any of the hypersurfaces considered (see \cite{PS} and examples in Section~\ref{section: examples} below).

However this is well-known for the specialists that Dubrovin--Frobenius manifold structure of $\H_{0; K}$ extends to the locus $t_{i,0} = t_{i,1} = \dots = t_{i, k_i} = 0$ for any $i > 0$.

\subsection*{Open WDVV equation}
Motivated by the studies of the open Gromov--Witten theories the new system of PDEs called \textit{open WDVV} was introduced in \cite{HS12}. 
For a fixed WDVV solution $F^c = F^c(t_1,\dots,t_n)$, the open WDVV equation  is the system of equations on a function $F^o = F^o(t_0,t_1,\dots,t_n)$ depending on an additional variable $t_0$, 
\begin{align}
\label{eq:open WDVV}
\sum_{\gamma,\delta=1}^n \frac{\d^3F^c}{\d t_\alpha\d t_\beta\d t_\mu}\eta^{\mu\nu}\frac{\d^2F^o}{\d t_\nu\d t_\gamma}+\frac{\d^2F^o}{\d t_\alpha\d t_\beta}\frac{\d^2F^o}{\d t_0\d t_\gamma}
=
\sum_{\gamma,\delta=1}^{n}\frac{\d^3F^c}{\d t_\gamma\d t_\beta\d t_\mu}\eta^{\mu\nu}\frac{\d^2F^o}{\d t_\nu\d t_\alpha}+\frac{\d^2F^o}{\d t_\gamma\d t_\beta}\frac{\d^2F^o}{\d t_0\d t_\alpha},
\end{align}
that should hold for any fixed $0 \le \alpha,\beta,\gamma \le n$.

Similarly to ``classical'' WDVV equation, open WDVV equation is associativity equation of a product defined by
\begin{equation}\label{eq: open extension product}
\frac{\p}{\p t_\alpha} \circ \frac{\p}{\p t_\beta} = \sum_{\gamma,\delta=1}^{n} \frac{\p^3 F^c}{\p t_\alpha \p t_\beta \p t_\gamma} \eta^{\gamma\delta} \frac{\p}{\p t_\delta} + \frac{\p^2 F^o}{\p t_\alpha \p t_\beta} \frac{\p}{\p t_0}, \quad 0 \le \alpha,\beta \le n.
\end{equation}

From this point of view the function $F^o$, solving open WDVV equation, defines an \textit{open extension} of a Dubrovin--Frobenius manifold given by $F^c$.

For $F^c$ being Dubrovin--Frobenius potential of type $A_n$ or $D_n$ the solutions to open WDVV equation were investigated in \cite{BB21,BB19}. In particular, for the $A_n$ and $D_n$ Dubrovin--Frobenius potentials $F^c_{A_n}$ and $F^c_{D_n}$ the authors constructed the functions $F^o_{A_n}$ and $F^o_{D_n}$, being solutions to open WDVV equation. The function $F^o_{A_n}$ appeared to be polynomial in $t_0,t_1,\dots,t_n$ and $F^o_{D_n}$ polynomial in $t_1,\dots,t_n$, but Laurent polynomial in $t_0$.

Additionally we assume that $F^o(t_0,\dots,t_n)$ is quasihomogeneous (see Section~\ref{section: main theorem}) and satisfies the conditions
\[\frac{\d^2 F^o}{\d t_{0}\d t_{n}}=0,\qquad \frac{\d^2 F^o}{\d t_{0,1} \d t_{N+1,0}}=1 \]
as it appears in open Gromov-Witten theory (cf. \cite{PST14}, \cite{BCT19}).

\subsection*{Open extensions of the Hurwitz--Frobenius manifolds}
Main theorems of this paper are the following
\\
\\
\begin{theorem}\label{theorem: main}
    For any $N \ge 0$ holds.
    \begin{enumerate}
        \item[(i)] Let $L := \lbrace k_0,\dots,k_N,k_{N+1} \rbrace$. Then  Dubrovin--Frobenius manifold product and pairing of $\H_{0; L}$ extend to the locus $t_{N+1,1} = \dots = t_{N+1,k_{N+1}} = 0$.
        \item[(ii)] Let $k_{N+1} \ge 2$. Then the product over this locus coincides with the open extension of the Hurwitz-Frobenius manifold $\H_{0; K}$ for $K = \lbrace k_0,\dots,k_N \rbrace$.

        Respective open WDVV solution is given by the pair
        \[
            (F^c, F^o) := \left( \F^K, \frac{\p \F^L}{\p t_{N+1, 1} } \mid_{t_{N+1,1} = \dots = t_{N+1,k_{N+1}} = 0} \right).
        \]
    \end{enumerate}
\end{theorem}

\begin{remark}
Note that part (i) above does not involve setting $t_{N+1,0}$ to zero.
\end{remark}

\begin{remark}
Note that the function $F^c$ in part (ii) above is the potential of the Hurwitz--Frobenius manifold with a shorter ramification profile rather than that used for $F^o$. We have by construction that $F^o$ depends also on the variable $t_{N+1,0}$ that is not a variable of $F^c$. This is exactly the additional variable of the open WDVV equation.
\end{remark}

Projections of the Dubrovin--Frobenius manifold structure to a submanifold were investigated in details by I. Strachan \cite{St01, St04}. However the structure we obtain in the theorem above is not Dubrovin--Frobenius.

Our next theorem provides an explicit formula for the open potential above.
\begin{theorem} \label{theorem: expl form}
    Denote $\widetilde{\lambda}(z,\bt) := \left.\lambda(z,\bt) \right|_{t_{N+1,a} = 0, \ a \ge 1}$. We have
    \begin{equation*}
        F^o = \left. \left(\int \widetilde{\lambda}(z,\bt) dz\right) \right|_{z = t_{N+1,0}} 
        + \sum_{\substack{\alpha_1,\dots,\alpha_{k_0-1} \ge 0 \\ \sum_{i=1}^{k_0-1} (k_0+1-i) \alpha_i = k_0+1}} (|\alpha|-2)! \prod_{i=1}^{k_0-1} \frac{t_{0,i}^{\alpha_i}}{\alpha_i!}.
    \end{equation*}
\end{theorem}

Related to our theorems above is the investigation of \cite{A26} and Theorem~1.4 therein. However one notes immediately two important differences. Our open WDVV solution is written down explicitly and unambiguously, while the solution of loc.cit. has the summand that needs to be found from the open WDVV. The second difference is that our construction works for all genus zero Hurwitz--Frobenius manifolds whereas the construction of \cite{A26} rely on a technical condition that yet has to be resolved (see equation (5) of Theorem 1.2 in loc.cit.).

\subsection*{Open extensions of A and D type Dubrovin-Frobenius manifolds}
It's well known that Dubrovin-Frobenius manifolds of ADE type are very special --- these are the only source of irreducible semisimple Dubrovin-Frobenius manifolds with the polynomial potentials (as was conjectured by B. Dubrovin and proved by C. Hertling in \cite{H}). 

A and D type Dubrovin--Frobenius manifolds can be considered in the setting of Hurwitz--Frobenius manifolds. In particular, $\H_{0;\{n+1\}}$ is isomorphic to $A_n$ Dubrovin--Frobenius manifold by the construction and $D_n$ Dubrovin--Frobenius manifold can be constructed in the following way.
Consider the space $\H^{even}_{0;\{2(n-1),2\}}$ of meromorphic functions $\lambda: \PP^1 \to \PP^1$ of the form 
\[
    \lambda(z) := \frac{z^{2(n-1)}}{n-1} + \sum_{\alpha=1}^{n-1}v_{\alpha}z^{2(\alpha-1)} - \frac{v_n^2}{2z^2}
\]
with arbitrary complex parameters $v_1,\dots,v_{n-1}$ and $v_n \in \CC^\ast$. Endow this space with the multiplication $\circ$ and pairing $\eta$ by the same residue calculus as in the definition of a Hurwitz--Frobenius manifold. It's folklore that one gets $D_n$ Dubrovin--Frobenius manifold in this way (we prove it in Section~\ref{section: D as HFM} for completeness). 

Our last theorem is the following.

\begin{theorem}\label{theorem: A and D}
    For any $m \ge 2$, applied to the spaces $\H_{0;\lbrace n+1,m \rbrace }$ and $\H_{0; \{2(n-1),2,m\}}^{even}$ the open WDVV solution of Theorem~\ref{theorem: main} coincides with the open extensions of $A_n$ and $D_n$ Dubrovin-Frobenius manifolds constructed in \cite{BB21}.
\end{theorem}

\subsection{Acknowledgements}
The research leading to these results has received funding from the Basic Research Program at the National Research University Higher School of Economics.

The first named author is grateful to Sergei Lando and Vadim Prokofiev for the fruitful discussions.
The authors are also grateful to Akiyoshi Sannai for his comments on the first version of this text.

\section{Hurwitz--Frobenius manifold}\label{section: HFM}

In this paper we will only deal with the Dubrovin--Frobenius manifolds constructed in a specific way - on the space of ramified coverings of $\PP^1$. These are called {\it Hurwitz--Frobenius manifold} after B. Dubrovin.

Consider the space of the meromorphic functions $\lambda(z)$ on $\CC$ of the form
\begin{equation*}
 \lambda(z) := \frac{z^{k_0}}{k_0} + \sum_{a=1}^{k_0-1} v_{0,a} z^{a-1} + \sum_{i=1}^N\sum _{a=1}^{k_i} \frac{v_{i,a}}{ \left(z-v_{i,0}\right)^{a}}
\end{equation*}
where $v_\bullet$ are complex parameters. Assume also that $k_0 \ge 2$.

Let $\mu := k_0-1 + N + \sum_{i=1}^N k_i$ be the number of parameters $v_\bullet$ above.
Denote
\[
    M := \lbrace v_{\bullet} \in \CC^\mu \ | v_{i,k_i} \neq 0 \ \text{ and } \ v_{i,0} \neq v_{j,0}\ \text{for} \ i \neq j \rbrace.
\]

Following B.~Dubrovin (see Lecture~5 of \cite{D2}) associate to $M$ the three-point function.
Denote by
\[
 \mathcal{R} := \lbrace \infty, v_{1,0},\dots, v_{N,0} \rbrace
\]
the set of ramification points of $\lambda$.
For $X \in \T_M$ let $X \cdot \lambda$ stand for the respective directional derivative. Set
\begin{equation}\label{eq: three-point function}
   \langle X,Y,Z \rangle := - \sum_{p \in \mathcal{R}} \res_{z=p} \Big( (X \cdot \lambda)(Y \cdot \lambda)(Z \cdot \lambda) \ \frac{dz }{\frac{\p \lambda}{\p z}} \ \Big), \quad \forall X,Y,Z \in \T_M,
\end{equation}

Define also the bilinear form $\eta: \T_M \otimes_M \T_M \to \O_M$ by
\[
    \eta(X,Y) := - \sum_{p \in \mathcal{R}} \res_{z=p}  \Big( (X \cdot \lambda)(Y \cdot \lambda) \ \frac{dz }{\frac{\p \lambda}{\p z}} \ \Big), \quad \forall X,Y,Z \in \T_M.
\]
This bilinear form is symmetric by the definition but also non--degenerate, defining the pairing on $\T_M$.

\begin{remark}
    At this point we make the choice of the so--called ``primary differential'' of B.~Dubrovin. In our case this is just $dz$.
\end{remark}

The three--point function and the pairing above allow one to introduce the product ${\circ: \T_M \otimes_M \T_M \to \T_M}$ by the equality
\[
    \eta(X \circ Y, Z) = \langle X,Y, Z \rangle, \quad \forall Z \in \T_M.
\]
This product introduces on $M$ a Dubrovin--Frobenius manifold structure (see Theorem~\ref{theorem: HFM} below). This follows immediately, that $e := \dfrac{\p}{\p v_{0,1}}$ is the unit of $\circ$ and $\eta(X,Y) = \langle e , X, Y \rangle$.

The following proposition will be our main computational tool for the product $\circ$. It can also be found in a bit different form in \cite{D2} Eq.(5.62).
\begin{proposition}\label{prop: product via lambda}
  Let $X,Y,Z \in \T_M$ be such that
   \[
      (X \cdot \lambda)(Y \cdot \lambda) - (Z \cdot \lambda) = \frac{\p \lambda}{\p z} \phi(z)
   \]
    assumed as the equality of rational functions, where $\phi(z)$ is a rational function such that $\sum_{p \in \mathcal{R}} \res_{z=p} ((W \cdot \lambda) \phi(z)dz) = 0$ for any $W \in \T_M$. 
    Then $X \circ Y = Z$.
\end{proposition}
\begin{proof}
    This follows immediately from the non-degeneracy of the pairing $\eta$.
\end{proof}

In what follows we will particularly use the proposition above with $\phi(z) = 1$.

\begin{proposition}\label{prop: SpecAn}
    The $\O_M$--algebra $(\T_M,\circ)$ of $\H_{0; K}$ is isomorphic to
    \[
         \A_K := \CC[z, w_1,\dots,w_N, y_1,\dots,y_N] \otimes \O_M / \I
    \]
    where $\I$ is the ideal generated by
    \begin{align}
        & P_{1,i} := y_i z - 1 - v_{i,0} y_i, \ P_{2,ij} := y_iy_j - \frac{y_i - y_j}{v_{i,0} - v_{j,0}}, \ P_{3,i} := w_i - \sum_{a=1}^{k_i} a v_{i,a} y_i^{a+1},
        \label{eq: specan-1}
        \\
        & P_{4,i} := w_i z - w_i v_{i,0} - \sum_{a=1}^{k_i} av_{i,a} y_i^{a},
        \
        P_{5} := z^{k_0-1} + \sum_{a=1}^{k_0-1} (a-1) v_{0,a} z^{a-2} - \sum_{i=1}^N w_i.
        \label{eq: specan-2}
    \end{align}
  The isomorphism is given by
    \[
        \Psi: y_i \mapsto \frac{\p }{\p v_{i,1}}, \ w_i \mapsto \frac{\p }{\p v_{i,0}}, \ z \mapsto \frac{\p }{\p v_{0,2}}.
    \]

\end{proposition}

\begin{proof}
    In this proof we apply Proposition~\ref{prop: product via lambda} multiple times.
    First of all note that $\frac{\p }{\p v_{0,1}}$ is the unit of $\circ$ and
    \[
        \frac{\p \lambda}{\p v_{i,1}} = \frac{1}{z - v_{i,0}},
        \
        \frac{\p \lambda}{\p v_{0,2}} = z.
    \]
    We have $\frac{\p }{\p v_{i,1}} \circ \frac{\p }{\p v_{0,2}} = \frac{\p }{\p v_{0,1}} +  v_{i,0} \frac{\p }{\p v_{i,1}}$ what gives the generator $P_{1,i}$.

    Similarly the second generator $P_{2,ij}$ originates from the simple equality
    \[
        \frac{1}{z - v_{i,0}}\frac{1}{z - v_{j,0}} = \frac{1}{v_{i,0} - v_{j,0}} \left( \frac{1}{z - v_{i,0}} - \frac{1}{z - v_{j,0}}\right).
    \]
    Next we have
    \[
     \frac{\p \lambda}{\p v_{i,0}} = \sum_{a=1}^{k_i} \frac{a v_{i,a}}{ (z-v_{i,0})^{a+1}}
    \]
    what gives the generators $P_{3,i}, P_{4,i}$.

    At the same time we have
    \[
        \frac{\p \lambda}{\p z} = z^{k_0-1} + \sum_{a=1}^{k_0-1} (a-1) v_{0,a} z^{a-2} - \sum_{i=1}^N \frac{\p \lambda}{\p v_{i,0}}.
    \]
    This gives the generator $P_5$ by Proposition~\ref{prop: product via lambda} again.

    In order to complete the proof we need to show that these generators compute any product of $\H_{0;K}$ after the map $\Psi$.

    We show that the quotient-ring above has the $\O_M$--basis
    \begin{equation}\label{eq: AK basis}
        [1],[z],\dots,[z^{k_0-2}], \ [w_i], \ [y_i],\dots,[y_i^{k_i}], \ i=1,\dots, N.
    \end{equation}
    First of all note that $P_{1,i}$ and $P_{4,i}$ allows us to express any product $[z^a y_i]$ via the basis elements above.

    \begin{lemma}\label{lemma_basis}
        For any $1 \le i,j \le N$ we can express $[w_iy_j]$ and $[w_iw_j]$ via basis elements above using the generators of $\I$.
    \end{lemma}
    \begin{proof}
        Let $i \neq j$.
        Consider
        \[
            P_{3,i} y_j = w_iy_j - \sum_{a=1}^{k_i} a v_{i,a} y_i^{a+1}y_j.
        \]
        Applying repeatedly the generator $P_{2,ij}$ we get the relation on $[w_iy_j]$ via the basis elements outlined above.

        In order to express $[w_iy_i]$ consider $P_5y_i$. The products $[z^{a-2}y_i]$ are simplified via the generator $P_{1,i}$ and the products $[w_iy_j]$ with $i \neq j$ by the reasoning above. This gives $[w_iy_i]$.

        The products $[w_iw_j]$ are now resolved with the help of the products computed above and the third generator $P_{3,i}$.
    \end{proof}

    To complete the proposition it remains to show that $[z^{a}]$ with $a \ge k_0-1$ and $[y_i^a]$ with $a \ge k_i+1$ are all expressible via the basis elements above.

    This is given by the generators $P_{3,i}$ and $P_5$.
\end{proof}

\begin{remark}
 Correct language for the algebra $\A_K$ above would be the analytic spectrum. This notion was introduced and investigated extensively by C. Hertling in \cite{H}. It is associated to any F--manifold and not even Dubrovin--Frobenius manifolds. However, we only use here the analytic spectra of the genus zero Hurwitz--Frobenius manifolds and do not want to introduce too many notations.
\end{remark}

The coordinates $v_\bullet$ look essential to parametrize the meromorphic functions with the prescribed ramification profile, however these are not flat for the pairing $\eta$ except for the simplest cases of low dimension. In order to define the potential we need to find the flat coordinates.

\subsection{Flat structure}\label{section: flat structure}

Let $z_0,z_i$ be such that $k_0 \lambda = z_0^{k_0}$, $\lambda = z_i^{k_i}$ in a neighborhood of $z = 0$ and $z = v_{i,0}$. We have
\[
  z_0 = z + v_{0,k_0-1}z^{-1} + \dots, \quad z_i = v_{i,k_i}^{1/k_i} (z - v_{i,0})^{-1} + \dots.
\]

Then define the coordinates $t_\bullet$ by
\begin{align}
    &t_{0,\alpha} = -\frac{1}{k_0-\alpha} \res_{z = \infty} z_0^{k_0-\alpha} dz,\
    \quad \alpha = 1,\dots,k_0-1,
    \label{eq: flat coordinates 0}
    \\
    &t_{i,0} = v_{i,0}, \  t_{i,1} = v_{i,1}, \quad
    t_{i,\alpha} = \frac{k_i}{k_i-\alpha + 1} \res_{z = v_{i,0}} z_i^{k_i-\alpha + 1} dz,\
    \quad \alpha = 2,\dots,k_i, \ i \ge 1. 
    \label{eq: flat coordinates i}
\end{align}

\begin{remark}\label{remark: flat coordinates are sector-wise}
    This follows immediately from the definition of $t_\bullet$ and $\lambda$ that $t_{i,\alpha} = t_{i,\alpha} (v_{i,0},\dots,v_{i,k_i})$. Namely, the flat coordinates are sector-wise functions of the $v_\bullet$ coordinates.

    This simple observation will play an important role later on.
\end{remark}

\begin{theorem}[Theorem 5.1 of \cite{D2}]\label{theorem: HFM}
    In the coordinates $t_\bullet$ above the three--point function and the pairing $\eta$ define the structure of a Dubrovin--Frobenius manifold on $\H_{0;K}$.

    The only non-zero pairings in the flat frame are
    \begin{align*}
        & \eta(\frac{\p}{\p t_{0,\alpha}},\frac{\p}{\p t_{0,k_0-\alpha}}) = 1, && 1 \le \alpha \le k_0-1,
        \\
        &\eta(\frac{\p}{\p t_{i,0}},\frac{\p}{\p t_{i,1}}) = 1, \
        \eta(\frac{\p}{\p t_{i,\alpha}},\frac{\p}{\p t_{i,k_i+1-\alpha}}) = \frac{1}{k_i}, && 2 \le \alpha \le k_i, \ i \ge 1.
    \end{align*}

    The three--point functions (cf. Eq.~\eqref{eq: three-point function}) are integrated by the potential $\F^K$
    \[
        \frac{\p^3 \F^K}{\p t_{i,\alpha} \p t_{j,\beta} \p t_{r,\gamma}} = \left\langle \frac{\p}{\p t_{i,\alpha}},\frac{\p}{\p t_{j,\beta}},\frac{\p}{\p t_{r,\gamma}} \right \rangle.
    \]

     The potential $\F^K$ is subject to the quasihomogeneity condition $E \cdot \F^K = 2(1 + \frac{1}{k_0}) \F^K$ modulo the quadratic terms. The Euler vector field $E$ reads
     \[
         E = \sum_{\alpha=1}^{k_0-1}
         \left( \frac{k_0 + 1}{k_0} - \frac{\alpha}{k_0} \right)
         t_{0,\alpha} \frac{\p}{\p t_{0,\alpha}} +  \sum_{i=1}^N \left( \frac{1}{k_0}t_{i,0} \frac{\d}{\d t_{i,0}}+\sum_{\alpha=1}^{k_i}
         \left(  \frac{k_0 + 1}{k_0} - \frac{\alpha-1}{k_i} \right) t_{i,\alpha} \frac{\p }{\p t_{i,\alpha}}  \right) 
     \]

\end{theorem}

\begin{example}
    For $K = \lbrace 5 ,5 \rbrace$ the flat coordinates are given by
    \begin{align*}
        & v_ {0, 4} =  t_ {0, 4}, \ v_ {0, 3} =  t_ {0, 3},  \ v_ {0, 2} =  t_ {0, 4}^2 + t_ {0, 2}, \ v_ {0, 1} =  t_ {0, 1} +  t_ {0, 3} t_ {0, 4},
        \\
        & v_ {1, 1} = t_ {1, 1}  \ v_ {1, 2} = \frac{1}{5}\left (t_ {1, 2} t_ {1, 5} + t_{1,3} t_ {1, 4} \right), \ v_ {1, 3} = \frac {1} {25} t_ {1, 5} \left (t_ {1, 4}^2 +
t_ {1, 3} t_ {1, 5} \right),
        \\
        & v_ {1, 4} = \frac {1} {125} t_ {1,4} t_ {1, 5}^3, \ v_ {1, 5} = \frac {t_ {1, 5}^5} {3125}, \ v_ {1, 0} =  t_ {1, 0}.
    \end{align*}
\end{example}

\section{Proof of Theorem~\ref{theorem: main}}\label{section: main theorem}
We prove main theorems of this paper in three steps following the numeration of the statements.

Fix $L := \lbrace k_0,\dots,k_N,k_{N+1} \rbrace$ and $K := \lbrace k_0,\dots,k_N\rbrace$. 
Let $M^L$ and $M^K$ be the respective spaces underlying the Hurwitz-Frobenius manifold structure (see Section~\ref{section: HFM}).

Denote by $\pi : M^L \to M^L$ the projection on the locus $v_{N+1,1} = \dots = v_{N+1,k_{N+1}} = 0$.
Let $M^{\res} := \pi(M^L)$.

\begin{proposition}\label{prop: product extends}
    Written in the basis $\frac{\p }{\p v_\bullet}$ the product of $\H_{0; L }$ extends to $M^\res$.
\end{proposition}
\begin{proof}
    We make the computations in the algebra $\A_L$ of Proposition~\ref{prop: SpecAn} and check that the products in question are polynomial in $v_{N+1,1}, \dots, v_{N+1,k_{N+1}}$.
    
    In the proof of Proposition~\ref{prop: SpecAn} we gave the procedure how to compute multiplication table in the basis~\eqref{eq: AK basis}. With our choice of the basis all coefficients $v_\bullet$ were coming to the multiplication table polynomially expect $v_{i,k_i}$ and $v_{i,0} - v_{j,0}$ --- first due to generator $P_{3,i}$ and second because of generator $P_{2,ij}$. We should take care of all applications of the generator $P_{3,i}$.
    Going through the proof of Proposition~\ref{prop: SpecAn} one notes that these are only applied to the following products below.

    \begin{lemma}\label{lemma: products with wN+1}
        We have in $\A_L$ for all $1 \le i \le N$
        \begin{align}
             & \pi_\ast ([w_{N+1} z^a]) = [w_{N+1} v_{N+1,0}^a],
             \quad
             \pi_\ast ([y_i^a w_{N+1}]) = \left[\frac{w_{N+1}}{(v_{N+1,0} - v_{i,0})^a}\right],
             \\
             & \pi_\ast( [w_i w_{N+1}]) = \left[\sum_{a=1}^{k_i} av_{i,a} \frac{w_{N+1}}{(v_{N+1,0} - v_{i,0})^{a+1}}\right],
             \\
             & \pi_\ast ([w_{N+1}^2]) = \left[w_{N+1} \left( v_{N+1,0}^{k_0-1} + \sum_{a=1}^{k_0-1}(a-1) v_{0,a} v_{N+1,0}^{a-2} - \sum_{i=1}^N \sum_{a=1}^{k_i} \frac{a v_{i,a}}{(v_{N+1,0} - v_{i,0})^{a+1}} \right)\right].
        \end{align}
    \end{lemma}
    \begin{proof}
        \textbf{Step 1}.
        $\pi_\ast ([w_{N+1} z^a])$ is expressed immediately via $P_{4,N+1}$. 
        
        \textbf{Step 2}.
        To compute $\pi_\ast ([y_i^a w_{N+1}])$ we start with $\pi_\ast ([y_i y_{N+1}^{a+1}])$.
        
        By $P_{2,i,N+1}$ we have $[y_i y_{N+1}^{a+1}] = \left[y_{N+1}^a \frac{y_{N+1} - y_i}{v_{N+1,0} - v_{i,0}}\right]$ and for some $R_a$ --- polynomials in $y_i$ and rational functions in $v_\bullet$ we have
        \begin{align*}
            [y_i y_{N+1}^{k_{N+1}+1}] &= \left[\frac{y_{N+1}^{k_{N+1}}}{v_{N+1,0} - v_{i,0}} + \sum_{a=0}^{k_{N+1}-1} y_{N+1}^{a} R_a \right]=
            \\
            & = \left[\frac{w_{N+1} - \sum_{a=1}^{k_{N+1}-1} a v_{N+1,a} y_{N+1}^{a+1}}{k_{N+1} v_{N+1,k_{N+1}} (v_{N+1,0} - v_{i,0})}\right] +  \left[\sum_{a=0}^{k_{N+1}-1} y_{N+1}^{a} R_a\right] \qquad (\text{by } P_{3,N+1}) .
        \end{align*}
        All the summands of the last expression project to zero by $\pi_\ast$.
        
        Now by $P_{3,N+1}$ again we have
        \[
            \pi_\ast ([y_i w_{N+1}]) = \pi_\ast ([ k_{N+1} v_{N+1,k_{N+1}} \cdot y_i y_{N+1}^{k_{N+1}} ]) = \left[\frac{w_{N+1}}{(v_{N+1,0} - v_{i,0})}\right].
        \]
        
        \textbf{Step 3}.
        $\pi_\ast ([w_i w_{N+1}])$ is resolved by $P_{3,i}$ giving the claim immediately.
        
        \textbf{Step 4}.
        Next we have 
        $\pi_\ast ([w_{N+1}^2]) = \pi_\ast ([w_{N+1} (z^{k_0-1} + \sum_{a=1}^{k_0-1} (a-1) v_{0,a} z^{a-2} - \sum_{i=1}^{N} w_i) ])$ by $P_5$. To finalize this expression we apply the computations of the previous steps.

    \end{proof}

\end{proof}
\begin{corollary}\label{corollary: w N+1 ideal}
 $\O_M \langle \frac{\p}{\p t_{N+1,0}} \rangle$ is an ideal of $\H_{0; L}$ product restricted to $M^\res$.
\end{corollary}
\begin{proof}
    The claim holds true because of the proposition above and the equality $t_{N+1,0} = v_{N+1,0}$.
\end{proof}

Corollary above was an important point of our investigation. This follows immediately from the definition of the open WDVV solution pair $(F^c,F^o)$ that the tangent space basis vector of the additional variable of $F^o$ generates an ideal. This was a suggestion for us that $t_{N+1,0}$ can play the role of an additional variable for some open WDVV solution.

\begin{corollary}\label{corollary: t locus}
    Part (i) of Theorem~\ref{theorem: main} holds true.
\end{corollary}
\begin{proof}
    It remains to check that the product of $\H_{0; L }$ extends to $M^\res$ in the flat 
    basis $\frac{\p }{\p t_\bullet}$. We recall that $t_{N+1,1}, \dots , t_{N+1,k_{N+1}}$ are functions of $v_{N+1,1}, \dots , v_{N+1,k_{N+1}}$ only (note that the variable $t_{N+1,0}$ is not included in this list). By the formulae~\eqref{eq: flat coordinates i}, the variables $t_{N+1, \alpha } = t_{N+1, \alpha }(v_\bullet)$ for $\alpha = 1, \dots, k_{N+1}-1$ have the poles only at $v_{N+1, k_{N+1}}$. To finish the proof, let us show that $t_{N+1, \alpha}(v_{N+1, \bullet})\mid_{v_{N+1,1} = \dots = v_{N+1,k_{N+1} -1} = 0}=0$ if $\alpha \neq 0,k_{N+1}$. By the definition, in the neighborhood of $v_{N+1, 0}$ we have
    \[
    z^{k_{N+1}}_{N+1}\mid_{v_{N+1,1} = \dots = v_{N+1,k_{N+1} -1} = 0} = \lambda\mid_{v_{N+1,1} = \dots = v_{N+1,k_{N+1} -1} = 0} = \frac{v_{N+1, k_{N+1}}}{(z-v_{N+1,0})^{k_{N+1}}}.
    \]
In what follows, under the restriction $v_{N+1,1} = \dots = v_{N+1,k_{N+1} -1} = 0$ we have
\[
t_{N+1, \alpha} =  \frac{k_{N+1}}{k_{N+1}-\alpha + 1} \res_{z = v_{N+1,0}}  \frac{(v_{N+1, k_{N+1}})^{\frac{k_{N+1} - \alpha +1}{k_{N+1}}}}{(z-v_{N+1,0})^{k_{N+1} - \alpha +1}}dz,
\]
for $\alpha \neq 0$. This residue is zero whereas $\alpha \neq k_{N+1}$.
\end{proof}
In order to prove the second part of the theorem we need to connect the products of the different Hurwitz space.

Let $\mathrm{pr}: M^\res \to M^\res$ be the projection to $v_{N+1,0} = 0$. Then $\mathrm{pr}(M^\res) \cong M^K$.
    
\begin{proposition}\label{prop: mt match}
    Let $\circ_L$ and $\circ_K$ stand for the products of $\H_{0,L}$ and $\H_{0,K}$ respectively.
    
    Then
    \[
        \mathrm{pr}_\ast \left(\frac{\p }{\p v_{i,\alpha}} \circ_L \frac{\p }{\p v_{j,\beta}} \right)
        =
        \mathrm{pr}_\ast \left(\frac{\p }{\p v_{i,\alpha}} \right) \circ_K \ \mathrm{pr}_\ast \left(\frac{\p }{\p v_{j,\beta}} \right)
    \]
    for $0 \le i,j \le N$ and arbitrary $\alpha,\beta$.
\end{proposition}
\begin{proof}
    In the $v_\bullet$ coordinates this follows from Corollary~\ref{corollary: w N+1 ideal}, Proposition~\ref{prop: SpecAn} and Proposition~\ref{prop: product extends}.

    In particular, we have that $\pi_\ast \A_L \cong \CC[z, w_1,\dots,w_N,w_{N+1}, y_1,\dots,y_N] \otimes \pi_\ast \O_M / \widetilde \I$ where $\widetilde \I$ is generated by $P_{1,i}, P_{2,ij}, P_{3,i}, P_{4,i}$, $i = 1,\dots, N$ and also
    \[
        z^{k_0-1} + \sum_{a=1}^{k_0-1} (a-1) v_{0,a} z^{a-2} - \sum_{i=1}^{N} w_i - w_{N+1}.
    \]

    The generators of $\widetilde I$ only differ from the generators of $\I$ in one additional appearance of $w_{N+1}$, however this appearance plays no role after $\mathrm{pr}_\ast$ by the corollary above.

    Since $t_{i,\alpha}$ and $t_{j,\beta}$ do not depend on $t_{N+1,0}$ and $t_{N+1,0} = v_{N+1,0}$, the statement also holds in the $t_{\bullet}$ coordinates.
    
\end{proof}

Let $I$ stand for the index set of $M^\res$ variables. We have
\[
    I = \left\lbrace (0, 1),\dots, (0, k_0-1), (1, 0) \dots, (N, k_N), (N+1, 0) \right\rbrace.
\]
\begin{proposition}\label{prop: FcFo last}
    We have for any $(i,\alpha), (j,\beta) \in I$
    \begin{align*}
        \mathrm{pr}_\ast \left( \frac{\p }{\p t_{i, \alpha}} \circ_L \frac{\p }{\p t_{j, \beta}} \right) & = \sum_{(m,\gamma), (n,\delta) \in I} \frac{\p^3 \F^{K}}{\p t_{i, \alpha} \p t_{j, \beta} \p t_{m, \gamma}} \eta^{m,\gamma ; n, \delta} \frac{\p }{\p t_{n, \delta}} 
        \\
        & + \left.\frac{\p^3 \F^L}{\p t_{i, \alpha} \p t_{j, \beta} \p t_{N+1, 1}}\right|_{t_{N+1,1} = \dots = t_{N+1,k_{N+1}} = 0} \frac{\p }{\p t_{N+1, 0}},
    \end{align*}
    where we assume the first summand on the right hand side to be zero if one of the indices is $(N+1,0)$.
\end{proposition}
\begin{proof}
    The product $\circ_L$ of $\H_{0; L}$  is integrated by the potential $\F^L$. Over $M^K$ this product coincides with the product of $\H_{0; K}$ by Proposition~\ref{prop: mt match}. Due to the specific form of the flat coordinates respective structure constants are given by the third derivatives of $\F^K$. This completes the proof.
\end{proof}

\begin{corollary}
    Part (ii) of Theorem~\ref{theorem: main} holds true.
\end{corollary}
\begin{proof}
    The pair of functions that is claimed to satisfy open WDVV equation integrates the product over $M^{res}$. 

    This product restricts for the product of $\H_{0; K}$ over $\mathrm{pr}(M^\res)$ by Proposition~\ref{prop: mt match}.

    By Proposition~\ref{prop: FcFo last} this product is integrated by two functions exactly as in Eq.~\eqref{eq: open extension product}. It is associative by the construction, what is equivalent to open WDVV equation.
\end{proof}

Let $\delta = 1 - \frac{2}{k_0}$ so that $E \cdot \F^{K} = (3-\delta) \F^{K}$ modulo the quadratic terms.
All solutions of the open WDVV equations, considered in the works~\cite{HS12,PST14,BCT19}, also satisfy the quasihomogeneity condition
\begin{gather}\label{eq:homogeneity for Fo}
E^\alpha\frac{\d F^o}{\d t_\alpha}+\frac{1-\delta}{2}t_0\frac{\d F^o}{\d t_0}=\frac{3-\delta}{2}F^o + \text{the linear terms in $t_{\bullet}$}
\end{gather}

The open extension potential of our theorem inherits quasihomogeneity property from the Hurwitz--Frobenius structure. 
\begin{proposition} \label{prop: quasihomogenity}
Let $(F^{c}, F^o)$ be a solution of the open WDVV equation from part (ii) of Theorem~\ref{theorem: main}.
    \begin{enumerate}
        \item[(i)] The equations
        \[\frac{\d^2 F^o}{\d t_{0,1}\d t_{n,\alpha}}=0,\qquad \frac{\d^2 F^o}{\d t_{0,1} \d t_{N+1,0}}=1 \]
hold for any $n \leq N$ and $\alpha \leq k_n$.
        \item[(ii)] $F^o$ satisfies the quasihomogeneity condition~\eqref{eq:homogeneity for Fo}
    \end{enumerate}
\end{proposition}
\begin{proof}
    Since $\F^{K}$ for $K = \{k_0,\dots,k_{N+1}\}$ is the potential of a Dubrovin-Frobenius structure, we have
\begin{align*}
   \frac{\d^2 F^o}{\d t_{0,1}\d t_{n,\alpha}} &=   \left.\frac{\d^3 \F^{K}}{\d t_{N+1,1} \d t_{0,1}\d t_{n,\alpha}}\right|_{t_{N+1,1} = \dots = t_{N+1,k_{N+1}} = 0} = \eta(\frac{\d}{\d t_{N+1,1}}, \frac{\d}{\d t_{n,\alpha}}) = 0, \quad \forall \alpha \leq k_n \\
   \frac{\d^2 F^o}{\d t_{0,1}\d t_{N+1,0}} &=   \left.\frac{\d^3 \F^{K}}{\d t_{N+1,1} \d t_{0,1}\d t_{N+1,0,}}\right|_{t_{N+1,1} = \dots = t_{N+1,k_{N+1}} = 0} = \eta(\frac{\d}{\d t_{N+1,1}}, \frac{\d}{\d t_{N+1,0}}) = 1, 
\end{align*} 
giving part $(i)$.

Denote by $\widetilde{E} = E|_{t_{N+1,1} = \dots = t_{N+1,k_{N+1}}=0}$. We have $[\widetilde{E}, \frac{\d}{\d t_{N+1,1}}] =0$. Also note that
\[
\left.\widetilde{E} \cdot \F^{K}\right|_{t_{N+1,1} = \dots = t_{N+1,k_{N+1}} = 0} = \left.E \cdot \F^{K}\right|_{t_{N+1,1} = \dots = t_{N+1,k_{N+1}} = 0} = \left.2(1 + \frac{1}{k_0})\F^{K}\right|_{t_{N+1,1} = \dots = t_{N+1,k_{N+1}} = 0}.
\]
Then 
\begin{align*}
&\widetilde{E}\cdot F^o = \widetilde{E} \cdot \left.\frac{\p \F^{K}}{\p t_{N+1, 1}} \right|_{t_{N+1,1} = \dots = t_{N+1,k_{N+1}} = 0} = \frac{\d}{\d t_{N+1,1}}\left.\left(\widetilde{E} \cdot \F^{K}\right)\right|_{t_{N+1,1} = \dots = t_{N+1,k_{N+1}} = 0} = \\
&=  \left.\frac{\d}{\d t_{N+1,1}}2(1 + \frac{1}{k_0})\F^{K}\right|_{t_{N+1,1} = \dots = t_{N+1,k_{N+1}} = 0} = 2(1 + \frac{1}{k_0}) \cdot F^o.
\end{align*}
 This gives part $(ii)$ and completes the proof.
\end{proof}

\section{Proof of Theorem~\ref{theorem: expl form}}\label{section: expl form}

In this section we find the explicit form of the open extension potential $F^o$.

For any $\A_L$--element $a$ denote by $\{w_{N+1}\} a$ the $[w_{N +1}]$ component of $a$ written in the basis~\eqref{eq: AK basis}.

\begin{proposition}\label{prop: more products}
    Let $i \neq j$ and $i,j \ge 1$. We have 
    \begin{align}
        & \{w_{N+1}\} ([z^{k_0-1+r}]) = v_{N+1,0}^r + \phi(v_{0,\bullet}), \ \{w_{N+1}\} ([w_i w_j]) = 0 , \ \{w_{N+1}\} ([z w_i ]) = 0 ,
        \label{eq: more products l1}
        \\
        & \{w_{N+1}\} ( [y_i^r w_j]) = 0, \ \{w_{N+1}\} ( [y_i^r y_j^s]) = 0, \ \{w_{N+1}\} ([y_i^r z^a]) = 0,
        \label{eq: more products l2}
        \\
        & \{w_{N+1}\} ([y_i^rw_i]) = - \frac{1}{(v_{N+1,0} - v_{i,0})^r},
        \ \{w_{N+1}\} ([w_i^2]) = - \sum_{a=1}^{k_i} \frac{a v_{i,a} }{(v_{N+1,0} - v_{i,0})^{a+1}},
        \\
        & \{w_{N+1}\} ([y_i^{k_i+1}]) = 0, \ \{w_{N+1}\} ([y_i^{k_i+2}]) = -\frac{1}{k_i v_{i,k_i}} \frac{1}{v_{N+1,0} - v_{i,0}},
    \end{align}
    where $\phi(v_{0,\bullet})$ is some polynomial in $v_{0,1},\dots,v_{0,k_0-1}$. 
\end{proposition}
\begin{proof}
    Modulo the addition of the polynomial terms in $v_{0,1},\dots,v_{0,k_0-1}$, by $P_5$ and $P_{4,N+1}$ we obtain $\{w_{N+1}\} ([z^{k_0-1+r}]) = \{w_{N+1}\} ([z^rw_{N+1}]) = v_{N+1,0}^r$. Then $P_{3,i}w_j = w_iw_j - \sum_{a=1}^{k_i} av_{i,a}y_i^{a+1}w_j$, and applying the same argument as in the proof of Lemma \ref{lemma_basis} we obtain $\{w_{N+1}\} ([w_i w_j]) = 0$. The last statement of the first line follows from the previous equality and from $P_{4,i}$.

    From $P_{3,j}y^r_i = y_i^rw_j - \sum_{a=1}^{k_i}y_{i}^{a+1}y_j^r$ and from $P_{2,ij}$ we conclude $\{w_{N+1}\} ( [y_i^r w_j]) = 0$ and $\{w_{N+1}\} ( [y_i^r y_j^s]) = 0$. The last statement of the second line follows from $P_{1,i}$.

    We have 
\[P_{5}y_i^r = z^{k_0-1}y_i^r + \sum_{a=1}^{k_0-1} (a-1) v_{0,a} z^{a-2}y_i^r - \sum_{b=1}^{N+1} w_by_i^r.
\]
Applying the previous line, we conclude $\{w_{N+1}\} ([y_i^rw_i]) = \{w_{N+1}\}([-w_{N+1}y_i^r]).$ Then $P_{3,N+1}y_i^r = w_{N+1}y_i^r + \sum_{a=1}^{k_{N+1}}av_{N+1,a}y_{N+1}^{a+1}y_i^r$, what gives us $\{w_{N+1}\}([-w_{N+1}y_i^r]) = \{w_{N+1}\}([\sum_{a=1}^{k_{N+1}}av_{N+1,a}y_{N+1}^{a+1}y_i^r])$. If $a < k_{N+1}$ we employ $P_{2,N+1;i}$ and these summands give a zero contribution, and for $a= k_{N+1}$ we apply the same expansion as in Step 2 of Lemma \ref{lemma: products with wN+1} what proves the claim. The last statement of the third line follows from $P_{3,i}$ applying the previous equality.

The last line follows from $P_{3,i}$ and $P_{3,i}y_i$ combining with the previous line.
    
\end{proof}

\begin{corollary}\label{corollary: v_{N+1,0} sc}
   For any $1 \leq i,j \leq N$, $i\neq j$ and arbitrary $\alpha, \beta$ the product of $\H_{0;L}$ reads 
   \[
   \frac{\d}{\d v_{i,\alpha}}\circ_{L} \frac{\d}{\d v_{j,\beta}} = \left.\frac{\d^2\lambda(z)}{\d v_{i,\alpha}\d v_{j,\beta}}\right|_{z=v_{N+1,0}}\frac{\d}{\d v_{N+1,0}} + \text{other terms}.
   \]
\end{corollary}
\begin{proof}
    This follows from the proposition above and the isomorphism $\Psi$ of Proposition~\ref{prop: SpecAn}. One notes that the structure constants of the right hand sides of the equalities coincide with the corresponding second order derivatives of $\lambda(z)\mid_{z=v_{N+1,0}}$.
\end{proof}

    Recall that $\widetilde{\lambda}(z,\bt) := \left.\lambda(z,\bt) \right|_{t_{N+1,a} = 0, \ a \ge 1}$.
        This follows from Proposition~\ref{prop: more products}, Corollary~\ref{corollary: v_{N+1,0} sc} and the equality $t_{N+1,0} = v_{N+1,0}$ that 
        \[
            \frac{\p^2 F^o}{\p t_{i,\alpha} \p t_{N+1,0}} = \left.\frac{\p \widetilde{\lambda}}{\p t_{i,\alpha}} \right|_{z = t_{N+1,0}}.
        \]
This means that $F^o = \int \widetilde{\lambda} dz \mid_{z = t_{N+1,0}} + A$, for some function $A$ depending on $t_{i,\alpha}$ with $i=0,1,\dots,N$ and independent of $t_{N+1,0}$. In order to prove the theorem we have to find $A$ explicitly.

This follows from Remark~\ref{remark: flat coordinates are sector-wise} and from Eq.~\eqref{eq: more products l2} that $\frac{\p^2 A}{\p t_{i,\alpha} \p t_{j,\beta}} = 0$ for $i \neq j$. 

If $i=j=0$, then from Remark~\ref{remark: flat coordinates are sector-wise} and Proposition~\ref{prop: more products} the second order derivatives in question coincide with those for $F_{A_{k_0-1}}^o$. Namely, we have
\[
\frac{\p^2 A}{\p t_{0,\alpha} \p t_{0,\beta}}= \left.\frac{\p^2 F^o}{\p t_{0,\alpha} \p t_{0,\beta}}\right|_{t_{N+1,0}=0} = \left.\frac{\p^2 F_{A_{k_0-1}}^o}{\p t_{0,\alpha} \p t_{0,\beta}}\right|_{t_{N+1,0}=0}  \]
Here the first equality holds since we take the derivation by the variables of the $0$--th sector. The second equality follows from Eq.~\eqref{eq: more products l1}, and from the fact that the formulae for flat coordinates $t_{0, \alpha}$ are the same for both $\H_{0;L}$ and $\H_{0; \{ k_0 \}}$.  

Now we show that $\frac{\p^2 A}{\p t_{i,\alpha} \p t_{i,\beta}} = 0$ for $i=j \ge 1$, what completes the proof. We have from open WDVV and part (ii) of Theorem~\ref{theorem: main} that 
        \[
            \left.\frac{\p^2 F^o}{\p t_{i,\alpha} \p t_{i,\beta}} \frac{\p \widetilde{\lambda}}{ \p z}\right|_{z = t_{N+1,0}} 
            = 
            \left.\frac{\p \widetilde{\lambda}}{\p t_{i,\alpha}} \frac{\p \widetilde{\lambda}}{\p t_{i,\beta}}\right|_{z = t_{N+1,0}} 
            - \left.\sum_{l, \gamma} c_{i,\alpha; i,\beta}^{l,\gamma} \frac{\p \widetilde{\lambda}}{\p t_{l,\gamma}} \right|_{z = t_{N+1,0}}.
        \]
where $ c_{i,\alpha; i,\beta}^{l,\gamma} = \sum_{j,\delta} \frac{\p^3F^{K}}{\p t_{i,\alpha } \p t_{i,\beta} \p t_{j,\delta}}\eta^{j,\delta;l,\gamma}$.

 Since  $\frac{\p \widetilde \lambda}{\p z} \not\equiv 0$  we conclude that
        \[
        \frac{\p^2 A}{\p t_{i,\alpha} \p t_{i,\beta}} = \res_{z = t_{i,0}} \left(\frac{\p \widetilde{\lambda}}{\p t_{i,\alpha}} \frac{\p \widetilde{\lambda}}{\p t_{i,\beta}}
            - \sum_{l,\gamma} c_{i,\alpha; i,\beta}^{l,\gamma} \frac{\p \widetilde{\lambda}}{\p t_{l,\gamma}} \right) \frac{dz}{(z - t_{i,0}) \cdot \frac{\p \widetilde{\lambda}}{\p z}} .
        \]
        However this follows from the definition of the Hurwitz--Frobenius manifold product that 
        \[
            \frac{\p \widetilde{\lambda}}{\p t_{i,\alpha}} \frac{\p \widetilde{\lambda}}{\p t_{i,\beta}} 
            - \sum_{l,\gamma} c_{i,\alpha; i,\beta}^{l,\gamma} \frac{\p \widetilde{\lambda}}{\p t_{l,\gamma}} = \frac{\p \widetilde{\lambda}}{\p z} \cdot \phi(z)
        \]
        for some function $\phi$ rational in $z$. One gets immediately from Proposition~\ref{prop: SpecAn} and Remark~\ref{remark: flat coordinates are sector-wise} that it has no constant term in the Laurent series expansion around $t_{i,0}$.

    This completed the proof.
    
\section{Examples}\label{section: examples}
We give several examples of the open WDVV solutions $(F^c,F^o)$ constructed above.

\subsection{Example 1}
Let $K = \lbrace 4,3 \rbrace$. We have
\begin{align*}
\F^{\lbrace 4, 3 \rbrace } &= \frac{1}{2} t_{0,1}^2 t_{0,3} + t_{0,1} \left( \frac{1}{2} t_{0,2}^2 + t_{1,0} t_{1,1} + \frac{1}{3} t_{1,2} t_{1,3} \right) + \frac{t_{0,3}^5}{60} - \frac{1}{4} t_{0,2}^2 t_{0,3}^2 \\
&\quad + t_{0,2} \left( \frac{1}{54} t_{1,3}^3 + \frac{1}{3} t_{1,0} t_{1,2} t_{1,3} + \frac{1}{2} t_{1,0}^2 t_{1,1} + t_{0,3} t_{1,1} \right) \\
&\quad + \left( \frac{1}{2} t_{1,0} t_{1,1} + \frac{1}{6} t_{1,2} t_{1,3} \right) t_{0,3}^2 + \left( \frac{1}{27} t_{1,0} t_{1,3}^3 + \frac{1}{3} t_{1,0}^3 t_{1,1} + \frac{1}{3} t_{1,0}^2 t_{1,2} t_{1,3} \right) t_{0,3} \\
&\quad + \frac{1}{54} t_{1,0}^3 t_{1,3}^3 + \frac{1}{20} t_{1,0}^5 t_{1,1} + \frac{1}{12} t_{1,0}^4 t_{1,2} t_{1,3} \\
&\quad + \frac{t_{1,1} t_{1,2}^2}{2 t_{1,3}} - \frac{t_{1,2}^4}{12 t_{1,3}^2} + \frac{1}{2} t_{1,1}^2 \log(t_{1,3})
\end{align*}

Then the open WDVV solution is given by 
\begin{align*}
F^c &= \frac {1}{2} t_{0,1}^2t_{0,3}+\frac{1}{2} t_{0,1} t_{0,2}^2  - \frac{1}{4} t_{0,2}^2 t_{0,3}^2 + \frac {t_ {0, 3}^5}{60},
\end{align*}
that is exactly the $A_3$ Dubrovin--Frobenius manifold potential
and
\begin{align*}
F^o & = \frac{t_{1,0}^5}{20} +\frac{1}{3} t_{0,3} t_{1,0}^3 +\frac{1}{2} t_{0,2} t_{1,0}^2 
+ \left(t_{0,1} + \frac{1}{2} t_{0,3}^2\right) t_{1,0} + t_{0,2} t_{0,3} .
\end{align*}
Compare this to the examples of \cite{BB21} Section~4.

\subsection{Example 2}
Let $K = \lbrace 4,2,2 \rbrace $. We have
\begin{align*}
\F^{\{4, 2, 2\}} &= \frac{1}{2} t_{0,1}^2 t_{0,3} + t_{0,1} \left[ \frac{1}{2} t_{0,2}^2 + \sum_{i=1}^2 \left( t_{i,0} t_{i,1} + \frac{1}{4} t_{i,2}^2 \right) \right] + \frac{t_{0,3}^5}{60} - \frac{1}{4} t_{0,2}^2 t_{0,3}^2 \\
&\quad + \sum_{i=1}^2 \Bigg[ t_{0,2} \left( \frac{1}{4} t_{i,0} t_{i,2}^2 + \frac{1}{2} t_{i,0}^2 t_{i,1} + t_{0,3} t_{i,1} \right)
+ \left( \frac{1}{2} t_{i,0} t_{i,1} + \frac{1}{8} t_{i,2}^2 \right) t_{0,3}^2
\\
&\quad \qquad + \left( \frac{1}{3} t_{i,0}^3 t_{i,1} + \frac{1}{4} t_{i,0}^2 t_{i,2}^2 \right) t_{0,3} + \frac{1}{20} t_{i,0}^5 t_{i,1} + \frac{1}{16} t_{i,0}^4 t_{i,2}^2 \Bigg] 
\\
&\quad + \frac{t_{1,2}^2 t_{2,1} - t_{1,1} t_{2,2}^2}{4(t_{1,0} - t_{2,0})} + \frac{t_{1,2}^2 t_{2,2}^2}{16(t_{1,0} - t_{2,0})^2} \\
&\quad + \sum_{i=1}^2 \frac{t_{i,1}^2}{2}  \log(t_{i,2}) + t_{1,1} t_{2,1} \log(t_{2,0} - t_{1,0}).
\end{align*}

Then the open WDVV solution is given by $F^c$ 
\begin{align*}
F^c &= \frac{1}{2} t_{0,1}^2 t_{0,3} + t_{0,1} \left( \frac{1}{2} t_{0,2}^2 + t_{1,0} t_{1,1} + \frac{1}{4} t_{1,2}^2 \right) + \frac{t_{0,3}^5}{60} - \frac{1}{4} t_{0,2}^2 t_{0,3}^2 \\
&\quad + t_{0,2} \left( \frac{1}{4} t_{1,0} t_{1,2}^2 + \frac{1}{2} t_{1,0}^2 t_{1,1} + t_{0,3} t_{1,1} \right) \\
&\quad + \left( \frac{1}{2} t_{1,0} t_{1,1} + \frac{1}{8} t_{1,2}^2 \right) t_{0,3}^2 + \left( \frac{1}{3} t_{1,0}^3 t_{1,1} + \frac{1}{4} t_{1,0}^2 t_{1,2}^2 \right) t_{0,3} \\
&\quad + \frac{1}{20} t_{1,0}^5 t_{1,1} + \frac{1}{16} t_{1,0}^4 t_{1,2}^2 
+ \frac{1}{2} t_{1,1}^2 \log(t_{1,2})
\end{align*}
and
\begin{align*}
F^o &= t_{0,1} t_{2,0} + t_{0,2} t_{0,3} 
+ \frac{1}{2} t_{0,3}^2 t_{2,0} + \frac{1}{2} t_{0,2} t_{2,0}^2 + \frac{1}{3} t_{0,3} t_{2,0}^3 + \frac{1}{20} t_{2,0}^5 
\\ 
&\quad 
+ \frac{t_{1,2}^2}{4(t_{1,0} - t_{2,0})} + t_{1,1} \log(t_{2,0} - t_{1,0}).
\end{align*}

\subsection{Example 3}
Let $K = \lbrace 4,3,3 \rbrace$. We have
\begin{align*}
\F^{\lbrace 4,3,3 \rbrace} &= \frac{1}{2} t_{0,1}^2 t_{0,3} + t_{0,1} \left[ \frac{1}{2} t_{0,2}^2 + \sum_{i=1}^2 \left( t_{i,0} t_{i,1} + \frac{1}{3} t_{i,2} t_{i,3} \right) \right] + \frac{t_{0,3}^5}{60} - \frac{1}{4} t_{0,2}^2 t_{0,3}^2 \\
&\quad + \sum_{i=1}^2 \Bigg[ t_{0,2} \left( \frac{1}{54} t_{i,3}^3 + \frac{1}{3} t_{i,0} t_{i,2} t_{i,3} + \frac{1}{2} t_{i,0}^2 t_{i,1} + t_{0,3} t_{i,1} \right) \\
&\quad \qquad \quad + \left( \frac{1}{2} t_{i,0} t_{i,1} + \frac{1}{6} t_{i,2} t_{i,3} \right) t_{0,3}^2 + \left( \frac{1}{27} t_{i,0} t_{i,3}^3 + \frac{1}{3} t_{i,0}^3 t_{i,1} + \frac{1}{3} t_{i,0}^2 t_{i,2} t_{i,3} \right) t_{0,3} \\
&\quad \qquad \quad + \frac{1}{54} t_{i,0}^3 t_{i,3}^3 + \frac{1}{20} t_{i,0}^5 t_{i,1} + \frac{1}{12} t_{i,0}^4 t_{i,2} t_{i,3} \Bigg] \\
&\quad + \sum_{i=1}^2 \frac{t_{i,1} t_{i,2}^2}{2 t_{i,3}} + \frac{t_{1,2} t_{1,3} t_{2,1} - t_{1,1} t_{2,2} t_{2,3}}{3(t_{1,0} - t_{2,0})} \\
&\quad - \sum_{i=1}^2 \frac{t_{i,2}^4}{12 t_{i,3}^2} + \frac{6 t_{1,2} t_{1,3} t_{2,2} t_{2,3} - t_{1,1} t_{2,3}^3 - t_{1,3}^3 t_{2,1}}{54(t_{1,0} - t_{2,0})^2} \\
&\quad + \frac{t_{1,2} t_{1,3} t_{2,3}^3 - t_{1,3}^3 t_{2,2} t_{2,3}}{81(t_{1,0} - t_{2,0})^3} - \frac{t_{1,3}^3 t_{2,3}^3}{486(t_{1,0} - t_{2,0})^4} \\
&\quad + \sum_{i=1}^2 \frac{1}{2} t_{i,1}^2 \log(t_{i,3}) + t_{1,1} t_{2,1} \log(t_{2,0} - t_{1,0}).
\end{align*}

Then the open WDVV solution is given by $F^c = \F^{\lbrace 4, 3 \rbrace }$ as in the previous section and
\begin{align*}
F^o &= t_{0,1} t_{2,0} + t_{0,2} t_{0,3} + \frac{1}{2} t_{0,3}^2 t_{2,0} + \frac{1}{2} t_{0,2} t_{2,0}^2 + \frac{1}{3} t_{0,3} t_{2,0}^3 + \frac{1}{20} t_{2,0}^5 
\\
&\quad + \frac{t_{1,2} t_{1,3}}{3(t_{1,0} - t_{2,0})} - \frac{t_{1,3}^3}{54(t_{1,0} - t_{2,0})^2} + t_{1,1} \log(t_{2,0} - t_{1,0})
\end{align*}

\section{Open potentials of type A and D}

The Dubrovin--Frobenius potentials $F_{A_n}$ and $F_{D_n}$ were first derived by B. Dubrovin via the geometry of the respective Coxeter groups (cf. \cite{D1}). 

We briefly recall the Dubrovin--Saito construction of a Frobenius manifold structure on the parameter space of a universal unfolding of $A$ and $D$ type singularities (see~\cite{ST}, \cite{BB21}). The {\it universal unfoldings} of the singularities $W=A_n$ and $W=D_n$ are the functions ${ \Lambda_{W}\colon\CC^{2}\times\CC^{n} \to \CC}$
\begin{align*}
    \Lambda_{A_n} &= \frac{x^{n+1}}{n+1} + y^2+ \sum_{k=1}^n s_k x^{k-1},\\
  \Lambda_{D_n} &= \frac{x^{n-1}}{n-1} + \frac{1}{2}xy^2 + \sum _{k=1}^{n-1} s_k x^{k-1} + s_{n}y.
\end{align*}
Consider the quotient ring
$$
\A_W:=\CC[x,y,s_1,\dots,s_n]/\left(\p_x\Lambda_W,\p_y\Lambda_W\right).
$$
As a $\mbC[s_1,\ldots,s_n]$-module, the space $\A_W$ has dimension $n$. Identifying $\mcT_{\mbC^n}$ and $\A_W$ via the isomorphism $\Phi_W$ defined by
\begin{gather*}
\Phi_W\left(\frac{\p}{\p s_k}\right) := \left[ \frac{\p \Lambda_W}{\p s_k}\right], \quad 1 \le k \le n,
\end{gather*}
we endow the tangent spaces $T_s \mbC^n$ with the multiplication. The structure constants $(c_s)_{i,j}^k$ of it are polynomials in the coordinates $s_1,\ldots,s_N$.  Define the bilinear form $\eta_W$ on~$\mbC^n$ by
$$
\eta_A\left(\frac{\d}{\d s_i},\frac{\d}{\d s_j}\right):=(c_s)^n_{i,j},
\quad
\eta_D\left(\frac{\d}{\d s_i},\frac{\d}{\d s_j}\right):=(c_s)^{n-1}_{i,j}
$$ 
This bilinear form together with the multiplication, introduced above, define a Dubrovin--Frobenius manifold structure on $\mbC^n$, often called the Saito Frobenius manifold. We denote by $F_{W}$ the potential of this manifold.

\subsection{$A_n$ and $D_n$ flat coordinates.} In \cite{NY} the authors gave the formulae for the flat coordinates of Dubrovin-Frobenius manifolds of $A_n$ and $D_n$ type via the unfolding coordinates $s_1,\dots,s_n$ of the corresponding singularities. 

For $A_n$ case we have 
\begin{align}\label{eq: A_n flat coordinates}
    t_\gamma = \sum_{\substack{\alpha_1,\dots,\alpha_n \ge 0 \\ \sum_{k=1}^{n} (n+2-k) \alpha_k = n+2 - \gamma}} \left(-1\right)^{|\alpha|-1} \prod_{k=0}^{|\alpha|-2}(\gamma + k(n+1)) \prod _{k=1}^{n} \frac{s_k^{\alpha _k}}{\alpha _k!},
\end{align}
where $|\alpha| = \sum_{k=1}^{n} \alpha_k$.

\begin{remark}\label{remark: 0-sector coordinates}
    The expressions $t_\gamma = t_\gamma(s_\bullet)$ above coincide with the expressions of the flat coordinates $t_{0,1},\dots,t_{0,n}$ of Hurwitz-Frobenius manifold $\H_{0;\{n+1,k_1,\dots,k_N\}}$ as the functions of $v_{0,1},\dots,v_{0,n}$ (see \cite{D2}, Exercise 5.3.) after the substitution $s_k \mapsto v_{0,k}$.
\end{remark}

For $D_n$ case we have 
\begin{align}
    t_{\gamma} &= \sum_{\substack{\alpha_1,\dots,\alpha_{n-1} \ge 0 \\ \sum_{k=1}^{n-1} (n-k)\alpha_k = n- \gamma}} 
    \left(-\frac{1}{2}\right)^{|\alpha|-1} 
    \prod _{k=0}^{|\alpha|-2} (2 \gamma -1 +2 k (n-1)) \prod_{k=1}^{n-1} \frac{s_k^{\alpha_k}}{\alpha_k!}, \quad 1 \le \gamma \le n-1, \label{eq: D flat coordinates}
    \\
    t_{n} &= s_n.
\end{align}
where $|\alpha| = \sum_{k=1}^{n-1} \alpha_k$.

\subsection{$A_n$ open potential}
The genus zero $A_n$ open potential $F^o_{A_n}$ was first found in \cite{BCT19}. It is a polynomial in $t_1,\dots,t_n$ and $t_0$ defined by
\begin{equation*}
    F^o_{A_n} := \frac{t_0^{n+2}}{(n+1)(n+2)}
    + \sum_{b=0}^n \sum_{\substack{\alpha_1,\dots,\alpha_n \ge 0 \\ \sum_{k=1}^{n} (n+2-k) \alpha_k = n+2 - b}} (|\alpha|+b-2)! \prod _{k=1}^{n} \frac{t_k^{\alpha _k}}{\alpha_k!} \cdot \frac{t_0^{b}}{b!}.
\end{equation*}
It is also connected to the $A_n$ unfolding coordinates $s^{A}_\bullet$ by the following identity (see \cite{Bu2})
\begin{equation}\label{eq: open potential derivative}
    \frac{\p F^o_{A_n}}{\p t_0} = \frac{t_0^{n+1}}{n+1} + \sum_{k=1}^n t_0^{k-1} s^{A}_k(t_1,\dots,t_n).
\end{equation}

\subsection{$D_n$ open potential}
The genus zero $D_n$ open potential $F^o_{D_n}$ is a polynomial in $t_1,\dots,t_n$ and Laurent polynomial in $t_0$ defined by (cf. \cite[Section 5.2]{BB21})
\begin{gather}\label{def: D_n open potential}
F^o_{D_n}:=\left.\left(\sum_{k=1}^{n-1} \frac{s_k s_{n+1}^{2 k-1}}{2^{k-1}(2 k-1)} + \frac{s_{n+1}^{2 n-1}}{2^{n-2}(2 n-1) (2n-2)} + \frac{s_n^2}{2 s_{n+1}}\right)\right|_{s_i=s_i(t)}.
\end{gather}
where the substitution of variables is given by inverting Eq.\eqref{eq: D flat coordinates} above and $s_n = t_n$.

It is immediate to see that the variable $t_n$ plays the special role in $F^o_{D_n}$. In particular, the only non--polynomial summand of $F^o_{D_n}$ is at the same time the only appearance of $t_n$ in the open potential. This variable is also special for $F_{D_n}$. In particular, we have (see \cite[Section 3.3]{BDbN})
\begin{equation}\label{eq: DN^c expression}
    F_{D_n} = \frac{1}{2}t_1^2t_{n-1} +\frac{1}{2}t_1 t_{n}^2 + \frac{1}{2} t_1 \sum_{\alpha=2}^{n-1} t_\alpha t_{n-\alpha} + \phi(t_2,\dots,t_{n-1}) + s_1^{\mathrm{D}} \cdot \frac{t_n^2}{2},
\end{equation}
for some polynomial $\phi$ that does not depend on $t_n$, and with the unfolding coordinates $s^{\mathrm{D}}_\bullet$.

\section{Proof of Theorem~\ref{theorem: A and D}}
In this section we prove Theorem \ref{theorem: A and D}. Namely, we show that our construction gives the open extensions of A and D type singularities, originally obtained in \cite{BB21}.

\subsection{Type A}
In \cite{BB21} the authors proved that if $F^o$ is a polynomial solution of open WDVV equations for $F^c=\F^{\{n+1\}} = F_{A_n}$ and satisfies Proposition \ref{prop: quasihomogenity}, then $F^o=F^o_{A_n}$. By Theorem \ref{theorem: expl form}, $F^o$ from part (ii) of Theorem \ref{theorem: main} is polynomial for $K=\{n+1\}$. We have also shown, that it satisfies Proposition \ref{prop: quasihomogenity}, then Theorem \ref{theorem: A and D} holds for the $A$ type singularity. In particular, \[
F^o_{A_n} = \left.\dfrac{\p \F^{\{n+1,m\}}}{\p t_{1, 1}} \right|_{t_{1,1} = \dots =t_{1,m}=0}.
\]

\subsection{Type D}\label{section: D as HFM}
First we are going to recover $D_n$ Saito--Frobenius manifold structure in the context of Hurwitz--Frobenius manifolds. This is a folklore, known for the specialists but we do not know any reference to cite except \cite{EY97} that seems to skip the proof. 

Consider the space $\H^{even}_{0; \{2n,2\} }$ of meromorphic functions on $\CC$ of the form 
\[
\lambda(z) := \frac{z^{2(n-1)}}{n-1} + \sum_{\alpha=1}^{n-1}v_{\alpha}z^{2(\alpha-1)} - \frac{v_n^2}{2z^2}
\]
with $v_n \neq 0$. Endow this space with the multiplication $\circ$ and pairing $\eta$ as in Section \ref{section: HFM}. Note that we make the specific choice of the scaling factors in $\lambda$ because it will be useful in the further statements.

Now we show, that $\H^{even}_{0; \{2(n-1),2\}}$ has the structure of Dubrovin-Frobenius manifold. 

\begin{proposition}\label{prop: AnSpec for D_n}
    $(\T_{\H^{even}_{2(n-1),2}}, \circ)$ is isomorphic to $\A_{D_n}$ by the map
\[
        \Psi: x \mapsto \frac{\p }{\p v_{2}}, \ y \mapsto \frac{\p }{\p v_{0}}, \ s_\alpha \mapsto v_\alpha.
\]
\end{proposition}
\begin{proof}
    It's enough to show that the relations of $\A_{D_n}$ hold true after the map $\Psi$. 
    We have
    \[
    \d_x \Lambda_{D_n} = x^{n-2} + \frac{1}{2}y^2 + \sum_{\alpha=1}^{n-1}(\alpha-1)s_\alpha x^{\alpha-2}, \quad \d_y \Lambda_{D_n} = xy + s_n.
    \]
We use Proposition \ref{prop: product via lambda} to prove the statement. We compute
\[
\frac{1}{z}\frac{\p \lambda}{\p z} = 2z^{2(n-2)} + \sum_{\alpha=1}^{n-1}2(\alpha-1)v_{\alpha}z^{2(\alpha-2)} + \frac{v_n^2}{z^4}.
\]
Since $\sum_{p\in\{\infty,0\}}\res_{z=p}\frac{dz}{z}= 0$, the right hand side of the above expression is zero in $(\T_{\H^{even}_{2n,2}}, \circ)$. This gives exactly the first relation of $\A_{D_n}$. Also note that
\[
\frac{\p \lambda}{\p v_n} \frac{\p \lambda}{\p v_2} = z^2\cdot (-\frac{v_n}{z^2}) = - v_n,
\]
what gives the second relation of $\A_{D_n}$.
\end{proof}

\begin{remark}
    One notes immediately that setting $v_n = 0$ to $\lambda$ above one recovers the unfolding $\Lambda_A$ with the even variables set to zero.
    This is in fact no miracle but the invariant sectors of the type D and type A Landau--Ginzburg orbifolds (see \cite{BR}).
\end{remark}

Now consider the potential $\F^{\{2(n-1),2\}}$ of the Hurwitz-Frobenius manifold $\H_{0;2(n-1),2}$. Define 
\[
\F_{\text{even}}^{\{2(n-1),2\}} := \F^{\{2(n-1),2\}}\mid_{t_{0,2\alpha =0}, \ \alpha \geq 1}.
\]

\begin{proposition}\label{prop: isomorphism to D_n}
    The function $\F_{\text{even}}^{\{2(n-1),2\}}$ defines the structure of Dubrovin-Frobenius manifold on the space $\H^{even}_{0; \{2(n-1),2\}}$. Moreover, it coincides with the potential $F_{D_n}$ after a rescaling of the variables. 
\end{proposition}
\begin{proof} From Proposition~\ref{prop: SpecAn} we conclude that the multiplication structure of $\H_{0; \{2(n-1),2\}}$ in the $v_{\bullet}$ coordinates does not have the poles at $v_{0,2\alpha}$ for any $\alpha$. It follows that the structure of product $\circ$ is correctly restricted to the locus $v_{0,2} = v_{0,4} = \dots = v_{0,2n-2} =0$. 

From Remark~\ref{remark: 0-sector coordinates} we obtain that $t_{0,2\gamma}$ vanishes under the restriction $v_{0,2} = v_{0,4} = \dots = v_{0,2n-2} =0$. If the summand in Equation~\ref{eq: A_n flat coordinates} does not contain $v_{0,2j}$ as a factor after the substitution $s_k \mapsto v_{0,k}$, then the left hand side of the equation $\sum_{k=1}^{2n-3} (2n-1-k) \alpha_k = 2n-1 - 2\gamma$ is even, whereas the right hand side is odd.

Note also that $\eta$ under this constraint is still non-degenerate, since $\eta(\frac{\p}{\p t_{0,\alpha}},\frac{\p}{\p t_{0,2n-2-\alpha}}) = 1$ are the only non-zero pairings. This shows that $\F_{\text{even}}^{\{2(n-1),2\}}$ is the solution to the WDVV equation and defines the structure of Dubrovin-Frobenius manifold. From Proposition \ref{prop: AnSpec for D_n} we conclude that this structure is isomorphic to Dubrovin-Frobenius manifold of $D_n$ singularity.
\end{proof}

Let $(F^c, F^o)$ be a solution of the open WDVV equation given by Theorem~\ref{theorem: main} for $F^c= \F^{\{2(n-1),2\}}$, and let this pair be restricted to the locus $t_{0,2\alpha}=0,\ \alpha \geq 1$. By this restriction, we move from the system of PDEs only the equations, where there are derivatives by the variables $t_{0,2\alpha}$. If all indices are odd, then the equations are still the same since $\eta$ pairs $\frac{\d}{\d t_{0,2\alpha-1}}$ only with $\frac{\d}{\d t_{0,2(n-1) - (2\alpha-1)}}$. Thus, the pair $\left.(F^c, F^o)\right|_{t_{0,2\alpha}=0,\ \alpha \geq 1}$ is also a solution of the open WDVV equation.

The following proposition finishes the proof of Theorem~\ref{theorem: A and D}.

\begin{proposition}\label{prop: open potential for the even space}
    Let $F^o$ be an open potential given by part $(ii)$ of Theorem~\ref{theorem: main} for ${F^c = \F^{\{2(n-1),2\}}}$. Then $\left.F^o\right|_{t_{0,2\alpha}=0,\ \alpha \geq 1}$ coincides with the open potential $F^{o}_{D_n}$ after the rescaling of the variables.
\end{proposition}
\begin{proof}
    By Theorem~\ref{theorem: expl form}, $F^{o} = (\int \widetilde{\lambda}(z,\bt)dz)|_{z=t_{2,0}} + A$, where $\widetilde{\lambda}(z,\bt) := \left.\lambda(z,\bt) \right|_{t_{2,a} = 0, \ a \ge 1}$ for
    \[
    \lambda (z, \bt)= \left.\left(\frac{z^{2(n-1)}}{n-1} + \sum_{\alpha=1}^{n-1}v_{0,\alpha}z^{2(\alpha-1)} - \frac{v_{1,2}^2}{2z^2} + \sum_{\alpha=1}^{m}v_{2,\alpha}(z-v_{2,0})^{-\alpha} \right)\right|_{v_{\bullet} = v_{\bullet}(t_{\bullet})}
    \]
    and
    \[
    A = \sum_{\substack{\alpha_1,\dots,\alpha_{k_0-1} \ge 0 \\ \sum_{i=1}^{2n-1} (2n+1-i) \alpha_i = 2n+1}} (|\alpha|-2)! \prod_{i=1}^{k_0-1} \frac{t_{0,i}^{\alpha_i}}{\alpha_i!}.
    \]

It follows from Definition \ref{def: D_n open potential} that $(\int \widetilde{\lambda}(z,\bt)dz)|_{z=t_{2,0}}$ coincides with the open potential $F^o_{D_n}$ after the rescaling of variables. Thus, it suffices to prove that $A |_{t_{0,2\alpha} = 0, \ \alpha \ge 1} = 0$.

If the summand of $A$ does not contain $t_{0,2j}$ as a factor, then the left hand side of the equation $\sum_{i=1}^{2n-3} (2n-1-i) \alpha_i = 2n-1$ is even, whereas the right hand side is odd. This gives the claim.
\end{proof}

\end{document}